\documentclass[11pt,twoside]{amsart}
\usepackage[all]{xy}   
        \usepackage {amssymb,latexsym,amsthm,amsmath,amsaddr}
        
\long\def\symbolfootnote[#1]#2{\begingroup%
\def\thefootnote{\fnsymbol{footnote}}\footnote[#1]{#2}\endgroup}

\newcommand{\Aut}{\textup{Aut}}

\makeatletter
\def\imod#1{\allowbreak\mkern10mu({\operator@font mod}\,\,#1)}
\makeatother

\newtheorem{theorem}{Theorem}[section]
\newtheorem{lemma}[theorem]{Lemma}
\newtheorem{corollary}[theorem]{Corollary}
\newtheorem{proposition}[theorem]{Proposition}

\newtheorem*{theorem*}{Theorem}
\theoremstyle{definition}

\numberwithin{equation}{section}

\newcommand{\ignore}[1]{}

\newcommand{\mynote}[1]{}
\begin{document}
\setcounter{section}{0}
\title{$z$-classes and rational conjugacy classes in alternating groups}
\author{Sushil Bhunia, Dilpreet Kaur, Anupam Singh}
\address{IISER Pune, Dr. Homi Bhabha Road, Pashan, Pune 411008 INDIA}
\email{sushilbhunia@gmail.com}
\email{dilpreetmaths@gmail.com}
\email{anupamk18@gmail.com}
\thanks{During this work, the first named author was supported by CSIR Ph.D. fellowship, and second named author was supported by NBHM Post-Doctoral fellowship.}
\today
\begin{abstract}
In this paper, we compute the number of $z$-classes (conjugacy classes of centralizers of elements) in the symmetric group $S_n$, when $n\geq 3$ and alternating group $A_n$ when $n\geq 4$. It turns out that the difference between the number of conjugacy classes and the number of $z$-classes for $S_n$ is determined by those restricted partitions of $n-2$ in which $1$ and $2$ do not appear as its part. And, in the case of alternating groups, it is determined by those restricted partitions of $n-3$ which has all its parts distinct, odd and in which $1$ (and $2$) does not appear as its part, along with an error term. The error term is given by those partitions of $n$ which have each of its part distinct, odd and perfect square. Further, we prove that the number of rational-valued irreducible complex characters for $A_n$ is same as the number of conjugacy classes which are rational.
\end{abstract}
\maketitle

\section{Introduction}
Let $G$ be a group. Two elements $x,y\in G$ are said to be $z$-conjugate if their centralizers $\mathcal Z_G(x)$ and $\mathcal Z_G(y)$ are conjugate in $G$. This defines an equivalence relation on $G$ and the equivalence classes are called $z$-classes. Clearly if $x$ and $y$ are conjugate then they are also $z$-conjugate. Thus, in general, $z$-conjugacy is a weaker relation than conjugacy on $G$. In the theory of groups of Lie type, this is also called ``types'' (see~\cite{gr}) and the number of $z$-classes of semisimple elements is called the genus number (see~\cite{ca1, ca2}). This has been studied explicitly for various groups of Lie type in several papers, see for example,~\cite{bs,go, gk, ku, si}. In this work, we want to classify and count the number of $z$-classes for symmetric and alternating groups. For convenience we deal with these groups when they are non-commutative (the commutative cases can be easily calculated), i.e., we assume $n\geq 3$ while dealing with symmetric groups and $n\geq 4$ while dealing with alternating groups. 

Let $\sigma\in S_n$. The conjugacy classes of elements in $S_n$ are determined by their cycle structure which, in turn, is determined by a partition of $n$. Let $\lambda = \lambda_1^{e_1} \lambda_2^{e_2}\cdots \lambda_r^{e_r}$ be a partition of $n$, i.e., we have $1\leq \lambda_1 < \lambda_2 < \ldots < \lambda_r\leq n$, each $e_i > 0$ and $n=\sum_{i=1}^r \lambda_ie_i$. We may represent an element of $S_n$ corresponding to a partition $\lambda$ in cycle notation. We prove the following,
\begin{theorem}\label{maintheoremsn}
Suppose $n\geq 3$. Let $\nu$ be a restricted partition of $n-2$ in which $1$ and $2$ do not appear as its part. Let $\lambda = 1^2\nu$ and $\mu=2^1\nu$ be partitions of $n$ obtained by extending $\nu$. Then the conjugacy classes of $\lambda$ and $\mu$ belong to the same $z$-class in $S_n$. Further, the converse is also true.
\end{theorem}
\begin{corollary}
The number of $z$-classes in $S_n$ is $p(n)-\tilde p(n-2)$, where $p(n)$ is the number of partitions of $n$ and $\tilde p(n-2)$ is the number of those restricted partitions of $n-2$ in which $1$ and $2$ do not appear as its part. Thus, the number of $z$-classes in $S_n$ is equal to $p(n)-p(n-2) + p(n-3) + p(n-4) - p(n-5)$.
\end{corollary}
\noindent To prove this theorem, we need to understand the centralizers better which involves the generalised symmetric group. A group $S(a,b)=C_a\wr S_b\cong C_a^b \rtimes S_b$ is called a generalised symmetric group. We will briefly introduce this group in the following section. 

Next we look at the problem of classifying $z$-classes in alternating groups $A_n$. Usually the conjugacy classes in $A_n$ are studied as a restriction of that of $S_n$. First, it is easy to determine for what partitions $\lambda=\lambda_1^{e_1}\cdots \lambda_r^{e_r}$ of $n$ the corresponding element $\sigma_{\lambda}$ is in $A_n$. This is precisely when $n-\sum e_i$ is even. We call such {\bf partitions even}. Further, when $\sigma_{\lambda} \in A_n$, the conjugacy class of $\sigma_{\lambda}$ in $S_n$ splits in two conjugacy classes in $A_n$ if and only if $\mathcal Z_{S_n}(\sigma_{\lambda}) = \mathcal Z_{A_n}(\sigma_{\lambda})$, which is, if and only if the partition $\lambda$ has all its parts distinct and odd, i.e., $e_i=1$ and $\lambda_i$ odd for all $i$. With this notation we have,
\begin{theorem}\label{maintheorem2}
Suppose $n\geq 4$. Let $\lambda=\lambda_1^{e_1}\cdots \lambda_r^{e_r}$ be an even partition of $n$. Then the following determines $z$-classes in $A_n$.
\begin{enumerate}
\item Suppose $e_i=1$ for all $i$ and all $\lambda_i$ are odd, i.e.,  $\lambda$ corresponds to two distinct conjugacy classes in $A_n$. Then, $\lambda$ corresponds to two distinct $z$-classes (corresponding to the two distinct conjugacy classes) if and only if all $\lambda_i$ are square. Else, the two split conjugacy classes form a single $z$-class.
\item Suppose either one of the $e_i\geq 2$ or at least one of the $\lambda_i$ is even, i.e., $\lambda$ corresponds to a unique conjugacy class in $A_n$. Then, $\lambda$ is $z$-equivalent to another conjugacy class if and only if $\lambda=1^3\nu$, where $\nu$ is a restricted partition of $n-3$, with all its parts distinct and odd, and in which $1$ (and $2$) does not appear as its part. Further the other equivalent class is $3^1\nu$.
\end{enumerate}
\end{theorem}
\noindent We remark that $3^1\nu$ could be of the first kind. For example, in $A_8$ the partitions $1^35^1$ and $3^15^1$ give same $z$-class. Further, the conjugacy class $3^15^1$ splits into two but both fall in a single $z$-class. We list few more examples (using GAP) in a table in Section~\ref{sum-squares}. We also note that $\nu$ could have its first part $3$, in that case while writing $3^1\nu$ we appropriately absorb the power of $3$. We denote by $\epsilon(n)$, the number of partitions of $n$ with all of its parts distinct, odd and square.  We list the values of $\epsilon(n)$ for small values in a table in Section~\ref{sum-squares}.
\begin{corollary}
The number of $z$-classes in $A_n$ is 
$$cl(A_n) - (q(n)+\tilde q(n-3)) + \epsilon(n),$$
where $cl(A_n)=\frac{p(n)+3q(n)}{2}$ is the number of conjugacy classes in $A_n$, $q(n)$ is the number of partitions of $n$ which has all parts distinct and odd, $\tilde q(m)$ is the number of restricted partitions of $m$, with all parts distinct, odd and which do not have $1$ (and $2$) as its part. 
\end{corollary}

Let $G$ be a finite group. An element $g\in G$ is called rational if $g$ is conjugate to $g^m$ for all $m$ with property $(m,o(g))=1$ where $o(g)$ is the order of $g$. Clearly if $g$ is rational then all of its conjugates are rational. Thus a conjugacy class of $G$ is said to be rational if it is a conjugacy class of a rational element. It is believed that, for a finite group $G$, the number of conjugacy classes which are rational is related to the number of rational-valued complex irreducible characters of the group $G$ (for example, see Theorem A in~\cite{nt}). A group of which all elements are rational (and in that case, all complex irreducible characters are rational-valued) is called a rational group or $\mathbb Q$-group (see~\cite{kl}). The alternating groups $A_n$ play an important role in determining simple groups which are rational (see Theorem A~\cite{fs}). There is a related notion of rational class in a group which comes from an equivalence relation. For a finite group $G$, a rational class of an element $g$ is a subset containing all elements of $G$ that are conjugate to $g^m$, where $(m,o(g))=1$. Thus the rational class of $g$ can be thought of as the conjugacy class of cyclic subgroup $\langle g \rangle$ of $G$. A conjugacy class which is rational is a rational class. However the converse need not be true. It is well known that, for a finite group $G$, the number of isomorphism classes of irreducible representations of $G$ over $\mathbb Q$ is equal to the number of rational classes of $G$ (see Corollary 1, Section 13.1~\cite{se}). The symmetric group $S_n$ is rational. Alternating groups are not rational (see Corollary B.1~\cite{fs}). The rational-valued complex irreducible characters for $A_n$ are discussed in~\cite{br} and~\cite{pr}. In this paper we determine conjugacy classes which are rational and the rational classes in alternating group. With notation as above,
\begin{theorem}\label{maintheorem3}
Suppose $n\geq 4$. Let $\tilde C$ be a conjugacy class in $A_n$ and corresponding partition be $\lambda=\lambda_1^{e_1} \cdots \lambda_r^{e_r}$ of $n$. 
\begin{enumerate}
\item 
Then the conjugacy class $\tilde C$ is rational in $A_n$ if and only if one of the following happens:  
\begin{enumerate}
\item either one of the $e_i\geq 2$ or one of the $\lambda_i$ is even, or,
\item when all $\lambda_i$ are distinct (i.e., $e_i=1$ for all $i$) and odd, the product $\displaystyle \prod_{i=1}^r \lambda_i$ is a perfect square. In this case, $\lambda$  corresponds to two conjugacy classes in $A_n$ and both are simultaneously rational (or non-rational).
\end{enumerate}
\item All conjugacy classes which are rational are rational classes. When $\tilde C$ is not a rational conjugacy class in $A_n$, the conjugacy class $C$ in $S_n$ containing $\tilde C$ is a rational class in $A_n$. 
\end{enumerate}
\end{theorem}
\noindent We denote by $\delta(n)$, the number of partitions of $n$ with all parts distinct, odd and the product of parts is a perfect square. We list the values of $\delta(n)$ for small values in a table in Section~\ref{sum-squares} which is also there in~\cite{br}. 
\begin{corollary}
For the alternating group $A_n$ with $n\geq 4$,
\begin{enumerate}
\item the number of conjugacy classes which are rational is $cl(A_n)-2q(n)+2\delta(n)$, and
\item the number of rational classes is $cl(A_n) - q(n) + \delta(n)$.
\end{enumerate}
\end{corollary}
\noindent The character theory of $A_n$ is well understood. We use the notation and results from~\cite{pr} and conclude the following,
\begin{theorem}\label{rational-character}
Suppose $n\geq 4$. Then, the number of conjugacy classes in $A_n$ which are rational is same as the number of rational-valued complex irreducible characters. 
\end{theorem}
\noindent This theorem is proved in Section~\ref{rational-charactersAn}. We also acknowledge that we have used GAP~\cite{GAP4} on several occasions to verify our computations and results.

{\bf Acknowledgement :} The authors would like to thank Gerhard Hiss and Alexander Hulpke for wonderful discussion on GAP during the workshop ``Group theory and computational methods" held at ICTS Bangalore, India in November 2016.  

\section{Restricted partitions}
We require certain kind of restricted partitions which we introduce in this section. We denote by $p(m)$, the number of partitions of positive integer $m$. To set the notation clearly, a partition of $m$ is $\lambda=m_1^{e_1}\cdots m_r^{e_r}$ where $1\leq m_1<\ldots < m_r \leq m$, $e_i\geq 1 \forall i$ and $m=\sum_{i=1}^r e_im_i$. Sometimes this is also denoted as $\lambda\vdash m$ or $m_1^{e_1}\cdots m_r^{e_r}\vdash m$. We clarify that the partition written as $1^12^1$ is same as $1.2$ but, in this case, latter notation is confusing if written without a dot. For us the significance of partitions is due to its one-one correspondence with conjugacy classes of the symmetric group $S_m$. Let $\tilde p(m)$ be the number of those partitions of $m$ in which $1$ and $2$ do not appear as its part, i.e., 
$$\tilde p(m)=\left |\left \{\lambda = m_1^{e_1}\cdots m_r^{e_r} \vdash m \mid  m_1 \geq 3 \right\} \right|.$$
Here we list down values of $\tilde p(m)$ for some small values.  
\vskip3mm
\begin{center}
\begin{tabular}{|c|c||c|c||c|c||c|c|}
\hline
$m$ & $\tilde p(m)$ &$m$& $\tilde p(m)$ & $m$& $\tilde p(m)$ &$m$ & $\tilde p(m)$\\ 
 \hline
1 & 0 & 6 & 2 & 11 &6 & 16 &21\\ \hline
2 & 0 & 7 & 2 & 12 & 9 &17 & 25 \\ \hline
3 & 1 & 8 & 3 & 13 & 10 &18 & 33\\ \hline
4 & 1 & 9 & 4 & 14 & 13 & 19 & 39 \\ \hline
5 & 1 & 10  & 5 & 15 & 17 & 20 & 49  \\ \hline
\end{tabular}
\end{center}
\vskip3mm
The generating function for $\tilde p(m)$ is 
$$\prod_{i\geq 3}\frac{1}{1-x^i}$$
and a formula to compute $\tilde p(m)$ in terms of partition function is
$$\tilde p(m)=p(m) - p(m-1) - p(m-2) + p(m-3).$$
This is a well known sequence in OEIS database (see~\cite{oeis}). This will be used in the study of $z$-classes of symmetric groups later.

Now we introduce the function $q(m)$. For a given integer $m$, the value of $q(m)$ is the number of those partitions of $m$ which have all of its parts distinct and odd, i.e., 
$$q(m)=|\{\lambda = m_1^{1}\cdots m_r^{1} \vdash m \mid m_i {\ \rm odd\ } \forall i\}|.$$
This number is same as the number of self-conjugate partitions.
For us this would correspond to those partitions which give split conjugacy classes in $A_n$. Now we introduce $\tilde q(m)$ which is the number of those restricted partitions of $m$ which have all its parts distinct, odd and $1$ (and $2$) does not appear as its part. The following table gives values of $\tilde q(m)$ for some values of $m$.
\vskip3mm
\begin{center}
\begin{tabular}{|c|c|c||c|c|c||c|c|c||c|c|c|}
\hline
$m$ & $q(m)$ &$\tilde q(m)$  & $m$ & $q(m)$ &$\tilde q(m)$ &$m$ & $q(m)$ &$\tilde q(m)$  & $m$ & $q(m)$ &$\tilde q(m)$\\  
 \hline
1 & 1 & 0 & 6 & 1 & 0& 11 &2 & 1 & 16 &5 &3 \\ \hline
2 & 0 & 0  & 7 & 1 & 1&12 &3 & 2 & 17 & 5& 2\\ \hline
3 & 1 & 1  & 8 & 2 & 1 &13 & 3 &1 & 18 & 5 &3\\ \hline
4 & 1 & 0 & 9 & 2 & 1 &14 & 3 &2 & 19 & 6 &3\\ \hline
5 & 1 & 1 & 10& 2 & 1& 15 & 4 &2 & 20 & 7 &4\\ \hline
\end{tabular}
\end{center}
\vskip3mm
The generating function for $q(m)$ is $\prod_{i\geq 0} (1+x^{2i+1})$ and the generating function for $\tilde q(m)$ is $\prod_{i\geq 1} (1+x^{2i+1})$.

\section{Symmetric groups}
In this section we classify $z$-classes in $S_n$. Since the centralizers are a product of generalised symmetric groups, we begin with a brief introduction to them.
\subsection{Generalised Symmetric Groups}\label{gsg}
The group $S(a,b)= C_a\wr S_b$, where $C_a$ is a cyclic group and $S_b$ is a symmetric group, is called a generalised symmetric group. This group is an example of wreath product and has been studied well in literature. Since the centralizer subgroups in the symmetric group are a product of generalised symmetric groups, we need to have more information about this group. For the sake of clarity, let us begin with defining this group. Consider the action of symmetric group $S_b$ on the direct product $C_a^b=C_a\times\cdots\times C_a$ given by permuting the components:
$$\sigma(x_1,\ldots, x_b)=(x_{\sigma(1)},\ldots, x_{\sigma(b)}).$$
Then the generalised symmetric group is $S(a,b)= C_a\wr S_b := C_a^b \rtimes S_b$.
Hence the multiplication in this group is given as follows:
$$(x_1,\cdots, x_b, \sigma)(y_1,\cdots,y_b, \tau) = (x_1y_{\sigma^{-1}(1)},\cdots, x_by_{\sigma^{-1}(b)}, \sigma\tau).$$
This group has a monomial matrix (each row and each column has exactly one non-zero entry) representation and it can be thought of as a subgroup of $GL_b(\mathbb C)$, in particular as a subgroup of monomial group. Monomial group is well known in the study of $GL_b(\mathbb C)$ as an algebraic group. This gives rise to the Weyl group and Bruhat decomposition. Let $T$ be the diagonal maximal torus (set of all diagonal matrices), then the monomial group is the normaliser $N_{GL_b(\mathbb C)}(T)$. The Weyl group is defined as $W=N_{GL_b(\mathbb C)}(T)/T\cong S_b$. 

Let $D$ be the set of those diagonal matrices in  $GL_b(\mathbb C)$ of which each diagonal entry is an $a$th roots of unity, i.e., each diagonal entry is from the set $\{\zeta^i\mid 0\leq i \leq a-1 \}$ where $\zeta$ is an $a$th primitive root of unity. Clearly, $D\cong C_a^b$ and the group $S(a,b)\cong N_{GL_b(\mathbb C)}(D)$. Thus, $S(a,b)$ is the set of those monomial matrices  which have non-zero entries coming from $a$th roots of unity. The following can be easily verified: 
\begin{enumerate}
\item the center $\mathcal Z(S(a,b))=\{\lambda.Id \mid \lambda^a=1\}\cong C_a$ if $a\geq 2$ or $b\geq 3$.
\item $N_{GL_b(\mathbb C)}(D)/D\cong S_b$.
\end{enumerate}
Representation theory of the generalised symmetric group has been studied by Osima~\cite{os}, Can~\cite{ca}, Mishra and Srinivasan~\cite{ms}, just to mention a few. 

\subsection{$z$-classes in $S_n$}\label{z-class-Sn}

In this section we aim to prove Theorem~\ref{maintheoremsn}. For $n=3$ and $4$ the conjugacy classes and $z$-classes are same. Thus, if necessary, we may assume $n\geq 5$ in this section. 
Let $\lambda = \lambda_1^{e_1} \lambda_2^{e_2}\cdots \lambda_r^{e_r}$ be a partition of $n$. Let us denote the partial sums as $n_i=\sum_{j=1}^{i}\lambda_je_j$ and $n_0=0$. We may represent an element of $S_n$ corresponding to $\lambda$ as a product of cycles and we choose a representative of class denoted as $\sigma_{\lambda}=\sigma_{\lambda_1} \cdots \sigma_{\lambda_i}\cdots \sigma_{\lambda_r}$ where 
$$\sigma_{\lambda_i}=\underbrace{(n_{i-1}+1,\cdots, n_{i-1}+\lambda_i)\cdots (n_{i-1}+(e_i-1) \lambda_i+1,\cdots, n_{i-1}+e_i \lambda_i)}_{e_i} $$
is a product of $e_i$ many disjoint cycles, each of length $\lambda_i$. 
Then the centralizer of this element is (see~\cite{jk} Equation 4.1.19)
$$\mathcal Z_{S_n}(\lambda):=\mathcal Z_{S_n}(\sigma_{\lambda}) \cong \prod_{i=1}^r C_{\lambda_i} \wr S_{e_i},$$ where $C_{\lambda_i}$ is a cyclic group of size $\lambda_i$ and the size of the centralizer is given by the formula $|\mathcal Z_{S_n}(\lambda)|= \prod_{i=1}^r (\lambda_i^{e_i} . e_i!)$. Further, with the above chosen representative element the center of $ \mathcal Z_{S_n}(\sigma_{\lambda})$ is, 
$$Z_{\lambda}=\mathcal Z (\mathcal Z_{S_n}(\sigma_{\lambda})) = \begin{cases} \prod_{i=1}^r \langle \sigma_{\lambda_i} \rangle & \mathrm{if}\ \lambda_1^{e_1}\neq 1^2 \\ \langle (1,2)\rangle \times \prod_{i=2}^r \langle \sigma_{\lambda_i} \rangle & \mathrm{when} \ \lambda_1^{e_1}= 1^2.
 \end{cases}$$
Note that if $\lambda_1=1$ then the element $\sigma_{\lambda_1}=1$.
\begin{lemma}\label{lemmar}
Let  $\lambda= \lambda_1^{e_1} \lambda_2^{e_2}\cdots \lambda_r^{e_r}$ be a partition of $n$. Then $\mathcal Z_{S_n}(\lambda)$ determines $r$ uniquely.
\end{lemma}
\begin{proof}
Consider the natural action of  $G=\mathcal Z_{S_n}(\lambda)$ on the set $\{1,2,\ldots, n\}$ as a subgroup of $S_n$. Since $G\cong \prod_{i=1}^r C_{\lambda_i} \wr S_{e_i}$, the orbits are $\{\{1,\ldots, n_1\}, \{ n_1+1, \ldots, n_2\},\ldots\}$. The number of orbits is exactly $r$. 
\end{proof}
\begin{lemma}\label{lemmapart}
Let  $\lambda= \lambda_1^{e_1} \lambda_2^{e_2}\cdots \lambda_r^{e_r}$ be a partition of $n$ and $\lambda_1^{e_1}\neq 1^2$. Let $Z_{\lambda}$ be the center of $\mathcal Z_{S_n}(\lambda)$. Then $Z_{\lambda}$ determines the partition $\lambda$ uniquely. 
\end{lemma}
\begin{proof}
Let us make $Z_{\lambda}$ act on the set $\{1,2,\ldots,n\}$. Then the orbits are of size $\lambda_i$ and each of them occur $e_i$ many times. This determines the partition $\lambda$.
\end{proof}
\begin{proposition}
Let $\lambda= \lambda_1^{e_1} \lambda_2^{e_2}\cdots \lambda_r^{e_r}$ and $\mu=\mu_1^{f_1} \mu_2^{f_2}\cdots \mu_s^{f_s}$ be partitions of $n$. Then $\mathcal Z_{S_n}(\lambda)$ is conjugate to $\mathcal Z_{S_n}(\mu)$ if and only if 
\begin{enumerate}
\item $r=s$,  
\item for all $i\geq 2$, $\lambda_i$ and $\mu_i$ are $\geq 3$ and $\lambda_i^{e_i} = \mu_i^{f_i}$,
\item $\lambda_1^{e_1}=1^2$ and $\mu_1^{f_1}=2^1$ or vice versa.
\end{enumerate}  
\end{proposition}
\begin{proof}
Clearly if the three conditions are given we have  $\lambda= 1^2\nu$ and $\mu=2^1\nu$ where $\nu=\nu_1^{l_1} \cdots \nu_k^{l_k}$ is a partition of $n-2$ with $\nu_1 > 2$. Thus the representative elements of the conjugacy classes are $\sigma_{\mu}=(12)\sigma_{\lambda}$ and $\sigma_{\lambda}$, where $\sigma_{\lambda}$ has cycles each of length $>2$. Thus centralizers of these two elements are same.

For the converse, we choose representative elements $\sigma_{\lambda}$ and $\sigma_{\mu}$ and we are given that $\mathcal Z_{S_n}(\sigma_{\lambda})$ and $\mathcal Z_{S_n}(\sigma_{\mu})$ are conjugate. The Lemma~\ref{lemmar} implies that $r=s$. Now we take the center of both of these groups $Z_{\lambda}$ and $Z_{\mu}$ and make it act on the set $\{1,2,\ldots,n\}$. If $\lambda_1^{e_1}$ and $\mu_1^{f_1}$ both are not $1^2$ then from Lemma~\ref{lemmapart} we get the required result.
\end{proof}
\noindent This proves Theorem~\ref{maintheoremsn}.

\section{Rational conjugacy classes in $A_n$}

The group $A_n$ is of index $2$ in $S_n$. Thus, we usually think of conjugacy classes in $A_n$ in terms of that of $S_n$. We have two kinds of conjugacy classes in $A_n$. Let $\sigma_{\lambda}$ be a representative of a conjugacy class, corresponding to a partition $\lambda$, of $S_n$. Suppose $\sigma_{\lambda}\in A_n$, that is to say, $\lambda$ is an even partition.  Then the two kinds of conjugacy classes are, 
\begin{enumerate}
\item[a.] {\bf Split:} The conjugacy class of $\sigma_{\lambda}$ in $S_n$ splits into two conjugacy classes in $A_n$ if and only if all parts of $\lambda$ are distinct and odd, which happens, if and only if $\mathcal Z_{S_n}(\sigma_{\lambda})=\mathcal Z_{A_n}(\sigma_{\lambda})$. 
\item[b.] {\bf Non-split:} The conjugacy class of $\sigma_{\lambda}$ remains a single conjugacy class in $A_n$ if and only if either one of the $e_i\geq 2$ for some $i$ or at least one of the $\lambda_i$ is even, which is, if and only if $\mathcal Z_{A_n}(\sigma_{\lambda})\subsetneq \mathcal Z_{S_n}(\sigma_{\lambda})$.
\end{enumerate}
While writing proofs in this section and later sections, we consider these two cases separately.

Let $G$ be a finite group and $g\in G$. The Weyl group of an element $g$ in $G$, denoted as, $W_{G}(g):=N_G(\langle g \rangle )/\mathcal Z_G(\langle g \rangle)$ where $\langle g \rangle$ is the subgroup generated by $g$. Using the map $\iota \colon N_G(\langle g \rangle) \rightarrow \Aut(\langle g \rangle)$ given by $\iota(x)(g^r)=xg^rx^{-1}$, one can show that, the element $g$ in $G$ is rational if and only if $W_G(g)\cong \Aut(\langle g \rangle)$. We need to understand Weyl group of elements $\sigma$ in $A_n$. Since $S_n$ is a rational group, we have, the Weyl group $W_{S_n}(\sigma) \cong \Aut(\langle \sigma \rangle)$. Thus, to understand if $\sigma$ is rational in $A_n$, we need to understand $N_{A_n}(\langle \sigma \rangle)$. This is determined by Brison (see Theorem 4.3~\cite{br}) as follows,
\begin{theorem}\label{brison}
Let $\sigma \in A_n$ and corresponding partition be $\lambda=\lambda_1^{e_1}\cdots \lambda_r^{e_r}$. Then, $N_{S_n}(\langle \sigma \rangle)= N_{A_n}(\langle \sigma \rangle)$ if and only if $\lambda$ satisfies the following,
\begin{enumerate}
\item all parts of $\lambda$ are distinct, i.e., $e_i=1$ for all $i$, 
\item $\lambda_i$ is odd for all $i$, and 
\item the product of parts $\displaystyle\prod_{i=1}^r \lambda_i \in \mathbb Z$ is a perfect square.
\end{enumerate}
\end{theorem}
\begin{corollary}\label{corollary_perfect_square}
Suppose $n$ is odd and $w=(1,2,\ldots,n)$ is in $A_n$. Then, $w$ is rational in $A_n$ if and only if $n$ is a perfect square (of odd number). 
\end{corollary}
\begin{proof}
We know $w$ is rational in $S_n$. Thus $W_{S_n}(w)\cong \Aut(\langle w \rangle)$. Since $n$ is odd the conjugacy class of $w$ in $S_n$ splits in $A_n$ and $\mathcal Z_{A_n}(w)=\mathcal Z_{S_n}(w)$. Thus $w$ is rational in $A_n$ if and only if $N_{S_n}(\langle w \rangle)=N_{A_n}(\langle w \rangle)$ which is if and only if $n$ is a perfect square (from Theorem~\ref{brison} above). 
\end{proof}
\noindent Now we determine which conjugacy classes are rational in $A_n$.

When the conjugacy class does not split, $C=\sigma_{\lambda}^{S_n}=\sigma_{\lambda}^{A_n}$ and $\mathcal Z_{A_n}(\sigma_{\lambda}) \subsetneq \mathcal Z_{S_n}(\sigma_{\lambda})$ is of index $2$. Then,
\begin{proposition}\label{rationalnsplit}
Let $C$ be a non-split conjugacy class in $A_n$. Then, $C$ is rational in $A_n$.
\end{proposition}
\begin{proof}
For this, we need to prove $\sigma_{\lambda}$ is conjugate to $\sigma_{\lambda}^m$ for all $m$ which is coprime to the order of $\sigma_{\lambda}$. Since $S_n$ is rational we have $g\in S_n$ such that $g\sigma_{\lambda}g^{-1}=\sigma_{\lambda}^m$. If $g$ is in $A_n$ we are done. Else take $h\in \mathcal Z_{S_n}(\sigma_{\lambda})$ which is not in $\mathcal Z_{A_n}(\sigma_{\lambda})$. Now $gh\in A_n$ and $gh\sigma_{\lambda}h^{-1}g^{-1}=\sigma_{\lambda}^m$, and we are done.
\end{proof}
\noindent When the conjugacy class splits, let $C$ be the conjugacy class of $\sigma_{\lambda}$ in $S_n$ where $\lambda = \lambda_1^1\cdots \lambda_r^1$ with all $\lambda_i$ odd (and distinct). Let $C_1$ and $C_2$ be the conjugacy classes in $A_n$, which are obtained by splitting $C$. Then,
\begin{proposition}\label{rationalsplit}
With the notation as above, both $A_n$ conjugacy classes $C_1$ and $C_2$ are rational if and only if $\displaystyle\prod_{i=1}^r \lambda_i$ is a perfect square.  
\end{proposition} 
\begin{proof}
In this case, we have $\mathcal Z_{A_n}(\sigma_{\lambda})= \mathcal Z_{S_n}(\sigma_{\lambda})$. Thus $W_{A_n}(\sigma_{\lambda})=W_{S_n}(\sigma_{\lambda})\cong \Aut(\langle \sigma_{\lambda}\rangle)$ if and only if $N_{A_n}(\langle \sigma_{\lambda}\rangle )= N_{S_n}(\langle \sigma_{\lambda}\rangle)$. Which is determined by Theorem~\ref{brison}.
\end{proof}
\noindent We remark that either both conjugacy classes $C_1$ and $C_2$ are rational or not rational simultaneously. 
\begin{proposition}\label{rationalclass}
With the notation as above, suppose both conjugacy classes $C_1$ and $C_2$ are not rational. Then the subset $C=C_1\cup C_2$ is a rational class in $A_n$. 
\end{proposition}
\begin{proof}
This follows easily because $C$ is a rational conjugacy class in $S_n$.
\end{proof}
\begin{proof}[{\bf Proof of Theorem~\ref{maintheorem3}}]
Let $\tilde C$ be a conjugacy class in $A_n$. Consider the conjugacy class $C$ in $S_n$ containing $\tilde C$. Let $\lambda=\lambda_1^{e_1}\cdots \lambda_r^{e_r}$ be the corresponding partition of $C$. Then either $\tilde C=C$ or $C=C_1 \cup C_2$, where $\tilde C$ is one of the $C_1$ or $C_2$. If $\tilde C=C$ it follows from Proposition~\ref{rationalnsplit} that it is always rational and this corresponds to the partitions where either $e_i\geq 2$ for some $i$ or one of the $\lambda_i$ is even.

Now suppose $C=C_1 \cup C_2$, where $\tilde C$ is one of the components. Then from Proposition~\ref{rationalsplit}, it follows that both $C_1$ and $C_2$, and hence $\tilde C$, are rational if and only if $\prod_{i=1}^r \lambda_i$ is a square. That is, in this case the partition $\lambda$ has all parts distinct, odd and the product of parts is a square.

When $\tilde C$ is not a rational conjugacy class, Proposition~\ref{rationalclass} implies $C$ is a rational class in $A_n$. This completes the proof.
\end{proof}

\section{$z$-classes in $A_n$ - when the conjugacy class splits}
Since the $z$-equivalence is a relation on conjugacy classes we deal with split and non-split classes separately. We begin with a few Lemmas.
\begin{lemma}\label{Centralizers_of_split_conjugacy_classes_in_A_n}
Let $x, y$ be elements in $A_n$ such that $x$ and $y$ are conjugate in $S_n$. If there exists $g\in A_n$ such that $g$ is conjugate to $y$ in $A_n$ and $\mathcal Z_{A_n}(g)=\mathcal Z_{A_n}(x)$ then centralizers $\mathcal Z_{A_n}(x)$ and  $\mathcal Z_{A_n}(y)$ are conjugate in $A_n$.
\end{lemma}
\begin{proof}
Since $g$ and $y$ are conjugate in $A_n$, their centralizers $\mathcal Z_{A_n}(g)$ and $\mathcal Z_{A_n}(y)$ are conjugate in $A_n$. Hence centralizers $\mathcal Z_{A_n}(x)$ and  $\mathcal Z_{A_n}(y)$ are conjugate in $A_n$.
\end{proof}
\noindent The following Lemma establishes partial converse to the above.
\begin{lemma}\label{converse}
Let $\lambda=\lambda_1\cdots\lambda_r$ be a partition with all parts (distinct and) odd. Let $x$ and $y$ be elements in $A_n$ representing the two distinct conjugacy classes corresponding to $\lambda$. Suppose $x=x_1x_2\cdots x_r$ and $y=y_1y_2\cdots y_r \in A_n$, where $x_i$ and $y_i$ are cycles of length $\lambda_i$ and centralizers $\mathcal Z_{A_n}(x)$ and $\mathcal Z_{A_n}(y)$ are conjugate in $A_n$. Then, $y$ is conjugate to $x_1^{i_1}x_2^{i_2}\cdots x_r^{i_r}$ in $A_n$ for some positive integers $i_1,\ldots, i_r$, where $i_j$ is coprime to $\lambda_j$ (which is the order of $x_j$) for all $j$.
\end{lemma}
\begin{proof}
Since $x$ and $y$ are $z$-conjugate in $A_n$, i.e., there exist $g\in A_n$ such that $\mathcal Z_{A_n}(x)=g\mathcal Z_{A_n}(y)g^{-1}=\mathcal Z_{A_n}(gyg^{-1})$. Now, we know that $\mathcal Z_{A_n}(x)= \mathcal Z_{S_n}(x)= \langle x_1, x_2, \ldots, x_r \rangle \cong C_{\lambda_1}\times\cdots \times C_{\lambda_r}$. Hence $gyg^{-1} = x_1^{i_1}x_2^{i_2}\cdots x_r^{i_r}$ for some $i_1,\ldots, i_r$. Therefore, $y$ is conjugate to $x_1^{i_1}x_2^{i_2}\cdots x_r^{i_r}$ in $A_n$.
\end{proof}
\noindent Now we prove the main proposition of this section.
\begin{proposition}~\label{split-z-class}
Let $\lambda=\lambda_1\cdots\lambda_r$ be a partition with all parts (distinct and) odd. Let $x$ and $y$ be elements in $A_n$ representing the two distinct conjugacy classes corresponding to $\lambda$. Suppose $x=x_1x_2\cdots x_r$ and $y=y_1y_2\cdots y_r$ written as a product of disjoint cycles where $x_i$ and $y_i$ are of length $\lambda_i$. Then, $x$ and $y$ are not $z$-conjugate in $A_n$ if and only if each $\lambda_i$ is a perfect square (of odd number) $\forall i=1,\ldots,r$.
\end{proposition}
\begin{proof}
First, suppose there exists a $k$ such that $\lambda_k$ is not a perfect square of odd number. We define $A_{\lambda_k}$ and $S_{\lambda_k}$ to be the subgroups of $A_n$ and $S_n$ respectively, on the symbols involved in the cycle $x_k$. 
Corollary~\ref{corollary_perfect_square} implies that the element $x_k$ is not a rational element of $A_{\lambda_k}$. Hence, there exists $m$ with $(m, \lambda_k)=1$ such that $x_k$ is not conjugate to $x_k^m$ in $A_{\lambda_k}$. In any case $x_k$ is conjugate to $x_k^m$ in $S_{\lambda_k}$, say, there exists $s\in S_{\lambda_k}\backslash A_{\lambda_k}$ such that $sx_ks^{-1}=x_k^m$. 
Thus, $sxs^{-1}=sx_1x_2\cdots x_k \cdots x_rs^{-1}=x_1x_2\cdots x_{k-1}x_k^mx_{k+1} \cdots x_r.$
We claim that $x$ is not conjugate to $sxs^{-1}$ in $A_n$. Because any two such elements will differ by an element of $\mathcal Z_{A_n}(x)$ which, in this case, is equal to $\mathcal Z_{S_n}(x)$ thus all such elements would be even. This implies that $x$ and $sxs^{-1}$ are representatives of the two distinct conjugacy classes obtained by splitting that of $x$ hence $sxs^{-1}$ is conjugate to $y$. But $\mathcal Z_{A_n}(x)= \langle x_1, x_2, \ldots, x_r \rangle = \langle x_1, x_2, \ldots, x_{k-1}, x_k^m, x_{k+1},\ldots, x_r \rangle = \mathcal Z_{A_n}(sxs^{-1})$, because of the structure of $sxs^{-1}$. Lemma~\ref{Centralizers_of_split_conjugacy_classes_in_A_n} implies that $\mathcal Z_{A_n}(x)$ and $\mathcal Z_{A_n}(y)$ are conjugate in $A_n$.
 
Now, assume that each $\lambda_i$ is a perfect square (of odd number), for all $i$. 
We define the subgroups $A_{\lambda_i}$ of $A_n$ on the symbols appearing in the cycle $x_i$ for all $i$. Corollary~\ref{corollary_perfect_square} implies that $x_i$ is rational in $A_{\lambda_i}$, hence $x_i$ is conjugate to $x_i^{m_i}$ in $A_{\lambda_i}$ for all $m_i$ with $(m_i, \lambda_i)=1$. Let $(j_1,\ldots, j_r)$ be a tuple where $(j_i,\lambda_i)=1$. Then we can find $s_{j_i}\in A_{\lambda_i}$ such that $s_{j_i}x_i s_{j_i}^{-1}=x_i^{j_i}$. 
Thus $s_{j_1}s_{j_2}\cdots s_{j_r}$ is in $A_n$  and conjugates $x$ to $x_1^{j_1}x_2^{j_2}\cdots x_r^{j_r}$. Hence, $y$ can not be conjugate to $x_1^{j_1}x_2^{j_2}\cdots x_r^{j_r}$ for any tuple $(j_1,\ldots, j_r)$ where $(j_i,\lambda_i)=1$. 
Lemma~\ref{converse}, implies that $x$ and $y$ can not be $z$-conjugate in $A_n$.
\end{proof}

\section{The center of centralizers in $A_n$}\label{center-an}
In Lemma~\ref{lemmapart} we showed that, for the group $S_n$, the center of centralizers $Z_{\lambda}$ determines the partition $\lambda$ uniquely via its action on the set $\{1,2,\ldots, n\}$ except in one case when $\lambda_1^{e_1}=1^2$. For the alternating groups we employ similar strategy. 

Let us begin with the case when the partition $\lambda$ has only one part, say, $\lambda=a^b$. The representative element can be chosen as follows,
$$\sigma_{\lambda}=(1,2,\cdots, a)(a+1,a+2,\cdots, 2a)\cdots ((b-1)a+1, (b-1)a+2,\cdots, ba )$$
which for convenience will be written as $\sigma_{\lambda}=\sigma_{\lambda,1}\sigma_{\lambda,2}\cdots\sigma_{\lambda,b}$ where $\sigma_{\lambda,i}$ are cycles of length $a$. And the centralizer is 
$$\mathcal Z_{S_n}(\sigma_\lambda) = \left(\langle (1,2,\cdots, a) \rangle \times \cdots \times \langle ((b-1)a+1, (b-1)a+2,\cdots, ba) \rangle\right) \rtimes S_b $$
where $S_b$ permutes the various cyclic subgroups. To avoid confusion, we write the elements of $S_b$ using roman numerals. For example, the element $(I,II)$ in $S_b$ would be actually $(1,a+1)(2,a+2)\cdots(a,2a)$ in $\mathcal Z_{S_n}(\sigma_\lambda)$, similarly, the element $(I,II,\cdots,b)$ in $S_b$ would be
$(1,a+1, \cdots, (b-1)a+1) (2,a+2, \cdots, (b-1)a+2)\cdots (a,2a, \cdots, ba)$. 
In general, the cycle $(I,II,\cdots,i)$ in $S_b$ would be  $(1,a+1, \cdots, (i-1)a+1) (2,a+2, \cdots, (i-1)a+2)\cdots (a,2a, \cdots, ia)$ which is a product of $a$ many disjoint cycles, each of length $i$. We can also compute $sgn((I,II\cdots,i))= sgn((1,2,\cdots,i))^{a}$ which will be useful to determine if $(I,II\cdots,i)$ belongs to $A_n$, when needed. 
We begin with,
\begin{lemma}\label{odd-perm}
If $\lambda=a^b$ is a partition of $n$ and $b\geq 2$ then $\mathcal Z_{S_n}(\sigma_{\lambda})$ contains at least one odd permutation.
\end{lemma}
\begin{proof}
If $a$ is even then the cycle $(1,2,\cdots,a) \in \mathcal Z_{S_n}(\sigma_{\lambda})$ is odd and we are done. Thus we may assume $a$ is odd. From the computation above, $sgn((I,II))= (-1)^{a}=-1$ hence $(I,II)$ is odd.
\end{proof}
\begin{lemma}\label{tau-non-diagonal}
\begin{enumerate}
\item If $\tau=(I,II,\cdots, b) \in \mathcal Z_{S_n}(\sigma_\lambda)$ then $\mathcal Z_{\mathcal Z_{S_n}(\sigma_\lambda)}(\tau) = <\tau, \sigma_{\lambda}>$,
\item If $\tau=(I,II,\cdots, b-1) \in \mathcal Z_{S_n}(\sigma_\lambda)$ then $\mathcal Z_{\mathcal Z_{S_n}(\sigma_\lambda)}(\tau) = <\tau, \sigma_{\lambda, b}, \prod_{i=1}^{b-1} \sigma_{\lambda, i}>$,
\end{enumerate}
\end{lemma}
\begin{proof}
The proof is simple and follows from the multiplication defined on $S(a,b)$ in Section~\ref{gsg}.  
\end{proof}
We need to understand the center of centralizers of elements in $A_n$. Suppose $\lambda= \lambda^{e_1}_1 \lambda^{e_2}_2 \cdots  \lambda^{e_r}_r$ is a partition of even type, i.e., $\sigma_{\lambda} \in A_n$.
Recall the notation, $\sigma_\lambda = \sigma_{\lambda_1} \cdots \sigma_{\lambda_r}$, the centralizer is $\mathcal Z_{S_n}(\sigma_\lambda) \cong \displaystyle\prod_{i=1}^r \mathcal Z_{S_{e_i\lambda_i}}(\sigma_{\lambda_i})$ and its center is denoted as $Z_{\lambda}$.
\begin{lemma}\label{CenterOfCentralizer}
Let $x\in A_n$. Then, $Z_{\lambda}\cap A_n = \mathcal Z(\mathcal Z_{S_n}(x))\cap A_n \subseteq \mathcal Z(\mathcal Z_{A_n}(x))$.
\end{lemma}
\begin{proof}
Let $g\in \mathcal Z(\mathcal Z_{S_n}(x))\cap A_n$ then $g\in \mathcal Z_{A_n}(x)$. Now $\mathcal Z_{A_n}(x)=\mathcal Z_{S_n}(x)\cap A_n$ thus we get $g\in \mathcal Z(\mathcal Z_{A_n}(x))$.
\end{proof}
\noindent Now we need to decide when $\mathcal Z(\mathcal Z_{A_n}(\sigma_\lambda))\supsetneq Z_\lambda \cap A_n$. For the convenience of reader we draw a diagram of the subgroups involved in the proofs. We call the elements of $Z_{\lambda}$ ``diagonal elements'' and the elements of $\mathcal Z_{S_n}(\sigma_{\lambda})$ which are not central ``non-diagonal elements''. 
\[\xymatrix@R=1pc{
& Z_{\lambda_1} \times \cdots \times Z_{\lambda_r}\ar@{=}[d]  &\ar@{=}[d] \mathcal Z_{S_{e_1\lambda_1}}(\sigma_{\lambda_1}) \times \cdots \times \mathcal Z_{S_{e_r\lambda_r}}(\sigma_{\lambda_r})& \\
Z_{\lambda}\ar@{-}^{=}[r]\ar@{-}[d] & \mathcal Z(\mathcal Z_{S_n}(\sigma_{\lambda})) \ar@{-}[r] & \mathcal Z_{S_n}(\sigma_{\lambda}) \ar@{-}[r]\ar@{-}[d] & S_n\ar@{-}[d] \\
Z_{\lambda}\cap A_n \ar@{-}^{\subset}[r]& \mathcal Z(\mathcal Z_{A_n}(\sigma_{\lambda})) \ar@{-}[r]& \mathcal Z_{A_n}(\sigma_{\lambda})  \ar@{-}[r] & A_n \\
}
\]
The main theorem is as follows,
\begin{theorem}\label{Exceptions}
Let $\lambda$ be a partition of $n$ and $\sigma_\lambda \in A_n$. Then $\mathcal Z(\mathcal Z_{A_n}(\sigma_\lambda))\supsetneq Z_\lambda\cap A_n$ if and only if $\lambda$ is one of the following:
\begin{enumerate}
\item $1^3\nu; 2^2\nu; 1^12^2\nu$ where $\nu=\lambda_3\cdots \lambda_r$ with all $\lambda_i\geq 3$ and odd.
\item $1^1\nu; \nu$ where $\nu=\lambda_3\cdots\lambda_{j-1}\lambda_j^{2}\lambda_{j+1}\cdots \lambda_r$ where $\lambda_i \geq 3$ and odd for all $i$.
\end{enumerate}
\end{theorem}
\noindent The rest of this section is devoted to the proof of this theorem. 
\begin{lemma}\label{RestrictionOnFixedPoints}
Let $\lambda=\lambda_1^{e_1} \cdots \lambda_r^{e_r}$ where at least two distinct $e_i$ and $e_j$ are $\geq 2$. 
Then, $\mathcal Z(\mathcal Z_{A_n}(\sigma_\lambda))= Z_\lambda \cap A_n$. 
\end{lemma}
\begin{proof}
Let us first take the case when $\lambda_1=1$, $e_1\geq 2$ and some other $e_i$ is $\geq 2$. We have $\mathcal Z_{S_n}(\sigma_\lambda)= S_{e_1}\times \mathcal Z_{S_{e_2\lambda_2}}(\sigma_{\lambda_2})\times \cdots \times \mathcal Z_{S_{e_r\lambda_r}}(\sigma_{\lambda_r})$. We note that when $e_1=2$ the subgroup $S_2$ is central. Let $g=(g_1,\ldots, g_r)\in \mathcal Z(\mathcal Z_{A_n}(\sigma_\lambda))$ but $g\notin Z_\lambda$. That is, there exists some $j$ such that $g_j$ is non-diagonal element in $\mathcal Z_{S_{e_j\lambda_j}}(\sigma_{\lambda_j})$. 

Suppose $j\neq 1$. Since $g_j$ is non-diagonal there exists $h_j\in \mathcal Z_{S_{e_j\lambda_j}}(\sigma_{\lambda_j})$ such that $h_jg_j\neq g_jh_j$. Now define $h=(1,\ldots,1,h_j,1,\ldots,1)$ if $h_j$ is even else $h= ((1,2),1,\ldots,1,h_j,1,\ldots,1)$. Then $h\in A_n \cap \mathcal Z_{S_n}(\sigma_{\lambda}) = \mathcal Z_{A_n}(\sigma_{\lambda})$ but $gh\neq hg$, a contradiction. 

Now if $j=1$ the element $g_1$ is non-diagonal in $S_{e_1}$, that is, $g_1\neq 1$. We may also assume that all other $g_i$, other than the first one, are diagonal. However if $e_1=2$ the element $g=((1,2),g_2,\ldots g_r)$ is already in $Z_\lambda$, so we couldn't have assumed otherwise. Now if $e_1\geq 3$, pick $h_1\in S_{e_1}$ which does not commute with $g_1$. Now define $h=(h_1,1,\ldots,1)$ if $h_1$ is even. Else define $h=(h_1,1,\ldots,1, w,1,\ldots,1)$ where $w\in \mathcal Z_{S_{e_i\lambda_i}}(\sigma_{\lambda_i})$ is an odd permutation guaranteed by Lemma~\ref{odd-perm}. Then $h\in A_n$ but $gh\neq hg$, a contradiction.

The proof when $e_1=1$ and $\lambda_1=1$ or $\lambda_1>1$ follows similarly. Now two components $i$ and $j$ will have odd elements because of Lemma~\ref{odd-perm} which can be used to change the sign to get an appropriate $h$.
\end{proof}
This reduces drastically the number of cases we need to look at. Thus we may assume that at most one $e_i$ is greater than $2$ or none, i.e., $\lambda=\lambda_1 \cdots\lambda_{i-1}\lambda_i^{e_i}\lambda_{i+1}\cdots \lambda_r$ with $e_i\geq 1$. Let us deal with the case when $i=1$ and $\lambda_1=1$.
\begin{lemma}\label{fixedPoints}
Let $\lambda =1^{e_1} \lambda_2\cdots \lambda_r$ be a partition of $n$. Then, $\mathcal Z(\mathcal Z_{A_n}(\sigma_\lambda))\supsetneq Z_\lambda \cap A_n$ if and only if $\lambda =1^{3} \lambda_2\cdots \lambda_r$ where $\lambda_i > 1$ and odd for all $i$.
\end{lemma}
\begin{proof}
Suppose $\lambda$ is not of the form $1^{3} \lambda_2\cdots \lambda_r$ where $\lambda_i > 1$ and odd for all $i$. So, if $e_1=0,1$ or $2$ then $\mathcal Z_{S_n}(\sigma_\lambda)$ is Abelian and its subgroup $\mathcal Z_{A_n}(\sigma_\lambda)$ is also Abelian. Therefore,
$$ \mathcal Z(\mathcal Z_{A_n}(\sigma_\lambda))=\mathcal Z_{A_n}(\sigma_\lambda)=\mathcal Z_{S_n}(\sigma_\lambda)\cap A_n = Z_\lambda \cap A_n.$$
Thus, we assume $e_1\geq 3$. Suppose at least one $\lambda_j$ is even.
Now $\sigma_\lambda$ has $e_1$ fixed points. Hence, $\mathcal Z_{S_n}(\sigma_\lambda)=S_{e_1}\times \langle \sigma_{\lambda_2} \rangle \times \cdots \times \langle \sigma_{\lambda_r} \rangle$ and $Z_{\lambda}= \mathcal Z(\mathcal Z_{S_n}(\sigma_\lambda))= \langle\sigma_{\lambda_2}\rangle \times \cdots \times \langle \sigma_{\lambda_r} \rangle$ since $\lambda_i$ are distinct. Now let $g\in \mathcal Z(\mathcal Z_{A_n}(\sigma_\lambda)) \subset \mathcal Z_{S_n}(\sigma_\lambda)$. Write $g=(g_1,g_2,\ldots,g_r)$. If $g\not\in Z_{\lambda}\cap A_n$ then $g_1\neq 1$. But we can find $h_1\in S_{e_1}$ such that $g_1h_1\neq h_1g_1$. Define $h=(h_1,1,\ldots,1)$ if $h_1$ is even else $h=(h_1,1, \ldots, 1, \sigma_{\lambda_j},1,\ldots,1)$. Clearly $h\in A_n \cap \mathcal Z_{S_n}(\sigma_\lambda) = \mathcal Z_{A_n}(\sigma_\lambda)$ and $gh\neq hg$. This contradicts that $g\in \mathcal Z(\mathcal Z_{A_n}(\sigma_\lambda))$, thus $\mathcal Z(\mathcal Z_{A_n}(\sigma_\lambda)) = Z_{\lambda} \cap A_n$.  

Now suppose $e_1\geq 4$ and all $\lambda_i$ are odd. In this case, $\mathcal Z_{A_n} (\sigma_\lambda) = A_{e_1}\times \langle \sigma_{\lambda_2}\rangle \times \cdots \times \langle \sigma_{\lambda_r}\rangle$ since all $\sigma_{\lambda_i}$ are even. And $\mathcal Z(\mathcal Z_{A_n}(\sigma_\lambda))= \langle \sigma_{\lambda_2}\rangle \times \cdots \times \langle \sigma_{\lambda_r} \rangle$ which is equal to $Z_\lambda\cap A_n$.
  
For the converse, $\lambda =1^3 \lambda_2 \cdots \lambda_r$ with all $\lambda_i$ odd. Then 
$\mathcal Z_{S_n}(\sigma_\lambda)=S_{3}\times \langle \sigma_{\lambda_2}\rangle \times \cdots \times \langle \sigma_{\lambda_r}\rangle$ and $Z_{\lambda}= \mathcal Z(\mathcal Z_{S_n}(\sigma_\lambda))= \langle\sigma_{\lambda_2}\rangle \times \cdots \times \langle \sigma_{\lambda_r}\rangle \subset A_n$. Also $\mathcal Z_{A_n}(\sigma_\lambda)=A_{3}\times \langle \sigma_{\lambda_2}\rangle \times \cdots \times \langle \sigma_{\lambda_r}\rangle = \mathcal Z(\mathcal Z_{A_n}(\sigma_{\lambda}))$. Hence $\mathcal Z(\mathcal Z_{A_n}(\sigma_\lambda))\supsetneq Z_\lambda \cap A_n$.
\end{proof}
\noindent This also takes care of the case when all parts are distinct so we may assume $e_i\geq 2$. Thus, assume either $i\geq 2$ or $\lambda_1\geq 2$. That is we can have at most one fixed point, if at all. If $\sigma_\lambda$ has one fixed point, say, $\sigma_{\lambda}(n)=n$ then we may consider $\sigma_{\lambda}\in A_{n-1}$ with no fixed points. Further $\mathcal Z_{A_n}(\sigma_{\lambda})= \mathcal Z_{A_{n-1}}(\sigma_{\lambda})$. Therefore, it is enough to study the partitions which do not have $1$ as its part, i.e., we have $\lambda_1>1$. 
\begin{lemma}\label{RestrictionOnNumberOfRepeatitions}
Let $\lambda= \lambda_1 \cdots\lambda_{i-1}\lambda_i^{e_i}\lambda_{i+1}\cdots \lambda_r$ be a partition with $\lambda_1>1$. Further suppose $\lambda$ satisfies one of the followings,
\begin{enumerate}
\item $e_i\geq 3$, or, 
\item $\lambda_i >2 $ and is even. 
\end{enumerate}
Then, $\mathcal Z(\mathcal Z_{A_n}(\sigma_\lambda))= Z_\lambda \cap A_n$.
\end{lemma}
\begin{proof} We prove (1) first.
Let $g=(g_1,\ldots, g_r)\in \mathcal Z(\mathcal Z_{A_n}(\sigma_\lambda))$ but $g\notin Z_\lambda$. 
All $g_j$ are diagonal except $g_i$ which is non-diagonal. The element $g_i\in \mathcal Z_{S_{e_i\lambda_i}}(\sigma_{\lambda_i})$ where $\sigma_{\lambda_i}=\sigma_{\lambda_{i,1}}\ldots \sigma_{\lambda_{i,e_i}}$. Recall the notation that $\sigma_{\lambda_{i,j}}$ are cycles of length $\lambda_i$ as introduced in the beginning of Section~\ref{center-an}. Now consider $\tau=(I,II,\ldots, e_i)$ as in Lemma~\ref{tau-non-diagonal}. Then $\mathcal Z_{\mathcal Z_{S_{e_i\lambda_i}}(\sigma_{\lambda_i})}(\tau)=\langle \sigma_{\lambda_i}, \tau \rangle$.
If $g_i=\tau$ then it does not commute with $h_i=\sigma_{\lambda_{i,1}}\sigma_{\lambda_{i,2}}$ (remember that $e_i\geq 3$). Since all $\sigma_{\lambda_{i,j}}$ are of same length $\lambda_i$ this is an element in $A_n$. Thus we get $h=(1,\ldots,1,h_i,1\ldots,1)$ in $\mathcal Z_{A_n}(\sigma_{\lambda})$ which does not commute with $g$, a contradiction. 

On other hand if $g_i\neq \tau $ then $g_i$ does not commute with $\tau$. We observe that $\tau$ is a product of $\lambda_i$ many cycles, each of length $e_i$. If $e_i$ is odd then $\tau$ is an even permutation. Further, if both $e_i$ and $\lambda_i$ are even then, also, $\tau$ is an even permutation. And in these cases we may take $h_i=\tau$ and get a contradiction as above. 
 
Now let us assume that $\lambda_i$ is odd and $e_i$ is even and thus $\tau$ is an odd permutation.  In this case instead of $\tau$ we make use of two elements $\tau_1, \tau_{2}\in \mathcal Z_{S_{e_i\lambda_i}}(\sigma_{\lambda_i})$ as follows. The element $\tau_1=(II, III, \ldots, e_i)$ and $\tau_2= (I,II, \ldots, e_i-1)$. Each of the $\tau_1$ and $\tau_2$ are product of $\lambda_i$ many cycles, each of length $e_i-1$ and hence even.  
Now we note that $\mathcal Z_{\mathcal Z_{S_{e_i\lambda_i}}(\sigma_{\lambda_i})}(\tau_1)=\langle \sigma_{\lambda_{i,1}}, \displaystyle \prod_{j=2}^{e_i} \sigma_{\lambda_{i,j}}, \tau_1 \rangle$ 
and $\mathcal Z_{\mathcal Z_{S_{e_i\lambda_i}}(\sigma_{\lambda_i})}(\tau_{2})=\langle \sigma_{\lambda_{i,e_i}}, \displaystyle \prod_{j=1}^{e_i-1}\sigma_{\lambda_{i,j}}, \tau_{2} \rangle$ (see Lemma~\ref{tau-non-diagonal}). This gives us that $\mathcal Z_{\mathcal Z_{S_{e_i\lambda_i}}(\sigma_{\lambda_i})}(\tau_1) \cap \mathcal Z_{\mathcal Z_{S_{e_i\lambda_i}}(\sigma_{\lambda_i})}(\tau_{2})=\langle \sigma_{\lambda_i} \rangle$. Since $g_i$ is non-diagonal it does not commute with either $\tau_1$ or $\tau_{2}$ else it would be in the intersection of centralizers which is diagonal. Thus we may take $h_i$ to be $\tau_1$ or $\tau_2$ as required, and get a contradiction.

For the proof of (2), let $g=(g_1,\ldots, g_r)\in \mathcal Z(\mathcal Z_{A_n}(\sigma_\lambda))$ but $g\notin Z_\lambda$. The component $g_i$ is non-diagonal element in $\mathcal Z_{S_{e_i\lambda_i}}(\sigma_{\lambda_i})$. In this case $\sigma_{\lambda_i}=\sigma_{\lambda_{i,1}}\sigma_{\lambda_{i,2}}$. Take $\tau=(I,II)$ then $\tau$ is an even permutation as $\lambda_i$ is even. If $\tau\neq g_i$ take $h_i=\tau$ and we are done. Else if $g_i=\tau$ then we take $\sigma_{\lambda_{i,1}}^2$. Since $\lambda_i>2$, $\sigma_{\lambda_{i,1}}^2\neq 1$ and it is even permutation. And now taking $h_i=\sigma_{\lambda_i,1}^2$ would lead to a contradiction.
\end{proof}
This leaves us with the following case now. The partition is $\lambda=\lambda_1 \cdots\lambda_{i-1}\lambda_i^{2}\lambda_{i+1}\cdots \lambda_r$ with $\lambda_1>1$ and either $\lambda_i=2$ or $\lambda_i$ is odd. And this is where all complication lies.
\begin{lemma}\label{EvenPartsNotAllowed2}
Let $\lambda$ with $\lambda_1>1$ be one of the followings,
\begin{enumerate}
\item $\lambda_1 \cdots\lambda_{i-1}\lambda_i^{2}\lambda_{i+1}\cdots \lambda_r$, and suppose, $\lambda_i$ is odd and $\lambda_m$ even for some $m\neq i$, or,
\item $2^2 \lambda_2 \cdots \lambda_r$ with some $\lambda_m$ even.
\end{enumerate}
Then, $\mathcal Z(\mathcal Z_{A_n}(\sigma_\lambda))= Z_\lambda \cap A_n$.
\end{lemma}
\begin{proof}
For the proof of (1), let $g=(g_1,\ldots, g_r)\in \mathcal Z(\mathcal Z_{A_n}(\sigma_\lambda))$ but $g\notin Z_\lambda$. Then $g_i$ is non-diagonal. Pick $h_i\in \mathcal Z_{S_{e_i\lambda_i}}(\sigma_{\lambda_i})$ such that $h_ig_i\neq g_ih_i$.  If $h_i$ is even then $h= (1,\ldots,h_i,\ldots,1)$ would do the job. Else take $h=(1,\ldots,h_i, 1, \ldots,\sigma_{\lambda_m}, 1\ldots,1)$ which is an even permutation, and does the job.

In the second case, we have $\sigma_\lambda=(1,2)(3,4)\sigma_{\lambda_2}\cdots \sigma_{\lambda_r}$ and $\mathcal Z_{S_n}(\sigma_\lambda) = \mathcal Z_{S_4}((1,2)(3,4))\times \langle \sigma_{\lambda_2} \rangle \times \cdots \times \langle \sigma_{\lambda_r} \rangle $.
Let $g=(g_1,\ldots, g_r)\in \mathcal Z(\mathcal Z_{A_n}(\sigma_\lambda))$ but $g\notin Z_\lambda$. In this case $g_1$ has to be non-diagonal. Now we can do the same thing as above to get a contradiction.
\end{proof}
At this step we are left with the $\lambda$ of following kinds, and its variant (see the discussion following Lemma~\ref{fixedPoints}) with exactly one fixed point, 
\begin{enumerate}
\item $\lambda_1 \cdots\lambda_{i-1}\lambda_i^{2}\lambda_{i+1}\cdots \lambda_r$, where all $\lambda_j$ are odd, and,
\item $2^2 \lambda_2 \cdots \lambda_r$, where all $\lambda_j$ are odd.
\end{enumerate}
Now we are ready to prove the main theorem of this section,
\begin{proof}[\bf{Proof of Theorem~\ref{Exceptions}}]
Lemma~\ref{RestrictionOnFixedPoints},~\ref{fixedPoints},~\ref{RestrictionOnNumberOfRepeatitions} and \ref{EvenPartsNotAllowed2} prove that if the partition $\lambda$ is not of the type listed in the theorem then $\mathcal Z(\mathcal Z_{A_n}(\sigma_\lambda))= Z_\lambda \cap A_n$. Thus it remains to prove if $\lambda$ is of one the kinds listed in the theorem then we do not get equality. Which we prove now case-by-case.
 
In case $\lambda=1^{3} \lambda_3\cdots \lambda_r$ and $\lambda_i$ are odd for all $i$ then the result follows from Lemma~\ref{fixedPoints}.
Now, take $\lambda=2^2 \lambda_3 \lambda_4 \cdots \lambda_r$ and $\lambda_3 \geq 2$ and odd for all $i$. Write $\sigma_\lambda=(1,2)(3,4)\sigma_{\lambda_3}\ldots\sigma_{\lambda_r}$ then $\mathcal Z_{S_n}(\sigma_{\lambda})= \{1,(1,2), (3,4), (1,2)(3,4), (1,3)(2,4), (1,3,2,4), (1,4,2,3), (1,4)(2,3)\} \times \langle \sigma_{\lambda_3} \rangle \cdots \times \langle \sigma_{\lambda_r} \rangle$. And $\mathcal Z_{A_n}(\sigma_{\lambda})= \{1,(1,2)(3,4), (1,3)(2,4), (1,4)(2,3)\} \times \langle \sigma_{\lambda_3} \rangle \cdots \times \langle \sigma_{\lambda_r} \rangle$ which is equal to its own center, being commutative. However the element $(1,3)(2,4)\notin Z_\lambda$. Thus we get strict inequality in this case. The argument is similar when $\lambda= 1^1 2^2 \lambda_3 \cdots \lambda_r$.
 
Now suppose $\lambda=\lambda_3 \cdots\lambda_{i-1}\lambda_i^{2}\lambda_{i+1}\cdots \lambda_r$ with $\lambda_3\geq 3$ and all odd. In this case,  $\sigma_\lambda = \sigma_{\lambda_3}\cdots\sigma_{\lambda_i} \cdots \sigma_{\lambda_r}$ where $\sigma_{\lambda_j}$
is a cycle of length $\lambda_j$ for $j\neq i$ and $\sigma_{\lambda_i} = \sigma_{\lambda_{i,1}} \sigma_{\lambda_{i,2}}$ is a product of two cycles, each of length $\lambda_i$. Then $\sigma_{\lambda_i,1}$ and $\sigma_{\lambda_i,2}$ both belong to $\mathcal Z(\mathcal Z_{A_n}(\sigma_\lambda))$ but none of them belong to $Z_{\lambda}$ instead their product belongs. A similar argument works for the case when $\lambda = 1^1 \lambda_3 \cdots \lambda_{i-1} \lambda_i^{2} \lambda_{i+1} \cdots \lambda_r$.
\end{proof}

\section{$z$-classes in $A_n$ - when the conjugacy class does not split}
Our strategy for the proof is similar to that of $S_n$ case. That is, we look at the action of $\mathcal Z(\mathcal Z_{A_n}(\sigma_{\lambda}))$ on $\{1,2,\ldots,n\}$ and decide when it determines the partition. This works in almost all cases. We continue to use notation from previous sections.
\begin{proposition}\label{Anpartition} 
The action of $Z_\lambda \cap A_n$ on the set $\{ 1, 2, \ldots, n\}$ determines the  partition $\lambda$ uniquely except when $\lambda_1^{e_1}=1^2$.
\end{proposition}
\begin{proof}
We know that the action of $Z_\lambda$ on the set $\{ 1, 2, \ldots, n\}$ determines the partition uniquely except when $\lambda_1^{e_1}=1^2$ (see Lemma~\ref{lemmapart}). We need to prove that if two points in $\{ 1, 2, \ldots, n\}$ are related under the action of $Z_\lambda$ then they are so under the action of $Z_\lambda \cap A_n$. 
 
Since $\sigma_\lambda = \sigma_{\lambda_1} \cdots \sigma_{\lambda_r}$, we reorder $\sigma_{\lambda_k}$'s, if required, so that $\sigma_{\lambda_k}$ for $1\leq k\leq l$ are even permutations and  $\sigma_{\lambda_k}$ for $l< k\leq r$ are odd permutations. Since $\sigma_\lambda$ is an even permutation, the number of odd permutations $r-l$ is even (including $0$). If $r=l$ then $Z_{\lambda}=Z_{\lambda}\cap A_n$ and we are done.
Else suppose $i\neq j$ are related under $Z_{\lambda}$. That is, there exists $t$ such that $\sigma_{\lambda_t}^m(i)=j$ for some power $m$. If $\sigma_{\lambda_t}^m$ is even, we are done. So we may assume $\sigma_{\lambda_t}^m$ is odd. But since the number of odd permutations is assumed to be even we have another odd permutation $\sigma_{\lambda_s}$ disjoint from this one. Thus, $\sigma_{\lambda_s}\sigma_{\lambda_t}^m$ will do the job. 
\end{proof}
\noindent We record the following example of the exception case. Take $\lambda=1^24^1\vdash 6$ then $\mathcal Z_{S_6}(3,4,5,6)=\langle (1,2)\rangle \times \langle (3,4,5,6)\rangle = Z_{1^24^1}$. And $Z_{1^24^1} \cap A_5=\langle (1,2) (3,4,5,6)\rangle$ which would determine the partition $2^14^1$. 
Now let us look at the case when $\mathcal Z(\mathcal Z_{A_n}(\sigma_\lambda))\neq Z_\lambda \cap A_n$. In this case we have the following,
\begin{proposition}\label{part-determine}
If $\lambda$ is one of the following with $\sigma_{\lambda}$ in $A_n$,
\begin{enumerate}
\item $1^2\lambda_2^{e_2}\cdots\lambda_r^{e_r}$, where $\lambda_2\geq 2$, or,
\item $1^1\nu, \nu$ where $\nu=\lambda_2 \cdots\lambda_{i-1}\lambda_i^{2}\lambda_{i+1}\cdots \lambda_r$, where $\lambda_j\geq 3$ and odd for all $j$,
\end{enumerate}
then, the action of $\mathcal Z(\mathcal Z_{A_n}(\sigma_\lambda))$ on the set $\{ 1,2, \dots, n\}$ determines the partition $\lambda$ uniquely.
\end{proposition}
\begin{proof}
The first case appears in $S_n$, where $\mathcal Z_{S_n}(\sigma_\lambda)$ determines all $\lambda_i>2$ except for the first orbit which is $\{1,2\}$. Thus there are two possibilities either $1^2$ or $2^1$. Since $\sigma_{\lambda} = \sigma_{\lambda_2}\cdots \sigma_{\lambda_r}\in A_n$ we note that the partition $2^1 \lambda_2^{e_2}\cdots\lambda_r^{e_r}$ is not even because this would correspond to the element $(1,2) \sigma_{\lambda} = (1,2)\sigma_{\lambda_2}\cdots \sigma_{\lambda_r}$ which is odd. Thus this leaves a unique choice for $\lambda$ where the first part must be $1^2$.
 
For the part (2), from the proof of Theorem~\ref{Exceptions}, we see that 
$$\mathcal Z(\mathcal Z_{A_n}(\sigma_\lambda)) = \langle \sigma_{\lambda_1},\ldots,\sigma_{\lambda_{i-1}}, \sigma_{\lambda_{i,1}},\sigma_{\lambda_{i,2}}, \sigma_{\lambda_{i+1}}, \ldots, \sigma_{\lambda_r}\rangle.$$
Clearly this determines the partition $\lambda$ uniquely.
\end{proof}
\noindent Now, we prove the main proposition as follows.
\begin{proposition}\label{main-prop-non-split}
Let $n\geq 4$. Let $\nu$ be a restricted partition of $n-3$, with distinct and odd parts, in which $1$ (and $2$) does not appear as its part. Let $\lambda = 1^3\nu$ and $\mu=3^1\nu$ be partitions of $n$ obtained by extending $\nu$. Then $\lambda$ and $\mu$ belong to the same $z$-class in $A_n$. Conversely, if $\lambda$ corresponds to a non-split class in $A_n$ then it can be $z$-equivalent to at most one more class (possibly split), provided $\lambda$ is of the form $1^3\nu$.
\end{proposition}
\begin{proof}
When the partition $\lambda=1^3 \nu$ then $\mathcal Z_{S_n}(\sigma_{\lambda})= S_3\times \mathcal Z_{S_{n-3}}(\sigma_{\nu})$ and its center is $Z_{\lambda}= \{1\} \times Z_{\nu}$. However $\mathcal Z_{A_n}(\sigma_{\lambda})= A_3\times \mathcal Z_{A_{n-3}}(\sigma_{\nu})$ is Abelian and its action would give the partition $3^1\nu$. In this case, if we take partition $\lambda'=3^1\nu$ then $\mathcal Z_{A_n}(\sigma_{\lambda'})=\langle (1,2,3) \rangle\times \mathcal Z_{A_{n-3}}(\sigma_{\nu})$ and $\mathcal Z_{A_n}(\sigma_{\lambda'}) = \mathcal Z_{A_n}(\sigma_{\lambda})$ (in case $\lambda'$ corresponds to a split class they are $z$-conjugate thus we may choose this representative). And thus $\sigma_{\lambda}$ and $\sigma_{\lambda'}$ would be $z$-conjugate.

For the converse, if $\mathcal Z(\mathcal Z_{A_n}(\sigma_\lambda))= Z_\lambda \cap A_n$, then from Proposition~\ref{Anpartition}, the action of $\mathcal Z(\mathcal Z_{A_n}(\sigma_\lambda))$ determines the partition $\lambda$ of $n$ uniquely, and we are done. Otherwise, we use Proposition~\ref{part-determine} which implies that $\mathcal Z_{A_n}(\sigma_\lambda)$ determines the partition $\lambda$ uniquely except in two cases. One of the cases is $1^3\nu$ where the centralizer is conjugate to that of $3^1\nu$ as required in the proposition. Thus we need to rule out the possibility when $\lambda=2^2\nu$ and $1^12^2\nu$ where $\nu=\lambda_3 \lambda_4 \cdots \lambda_r$, $\lambda_i\geq 3$ and are odd for all $i$. 

Let us deal with the case when $\lambda=2^2\nu$, the other case is similar. The element $\sigma_\lambda=(1,2)(3,4)\sigma_{\lambda_2}\cdots \sigma_{\lambda_r}$ and $$\mathcal Z_{A_n}(\sigma_\lambda)= \langle (1,2)(3,4), (1,3)(2,4)\rangle \times \langle \sigma_{\lambda_2}\rangle \times \cdots \times \langle \sigma_{\lambda_r} \rangle$$
which has size $4.\lambda_2.\cdots.\lambda_r$.
Since this is Abelian its center is itself which determines the partition $\lambda$ except for the first orbit which is $\{1,2,3,4\}$. Considering that $\lambda$ is even, we have the possibilities of the first part being $1^4, 1^13^1,2^2$. We claim that if $\lambda=1^4\nu, 1^13^1\nu$ or $2^2\nu$ the size of centralizers is different and hence they can not be $z$-equivalent. We note that, $\mathcal Z_{A_n}(\sigma_{1^4\nu})= \langle A_4 \rangle \times \langle \sigma_{\lambda_2}\rangle \times \cdots \times \langle \sigma_{\lambda_r} \rangle$ which has size $12.\lambda_2.\cdots.\lambda_r$. And if $\lambda_2>3$, $\mathcal Z_{A_n}(\sigma_{1^13^1\nu})= \langle (1,2,3) \rangle \times \langle \sigma_{\lambda_2}\rangle \times \cdots \times \langle \sigma_{\lambda_r} \rangle$ of size $3.\lambda_2.\cdots.\lambda_r$ and if $\lambda_2=3$, $\mathcal Z_{A_n}(\sigma_{1^13^2\lambda_3\cdots\lambda_r})= \langle (2,3,4), (5,6,7) \rangle \times \langle \sigma_{\lambda_3}\rangle \times \cdots \times \langle \sigma_{\lambda_r} \rangle$ of size $3^2.\lambda_3.\cdots.\lambda_r$. 
\end{proof}

\subsection{Proof of Theorem~\ref{maintheorem2}}
Let $C$ be a conjugacy class of $S_n$ corresponding to the partition $\lambda_1 \lambda_2\cdots \lambda_r$ of $n$ with all $\lambda_i$ distinct and odd. Then the conjugacy class $C$ splits in two conjugacy classes, say, $C_1$ and $C_2$ in $A_n$.
From Proposition~\ref{split-z-class} if each $\lambda_i$ is a perfect square for all $1\leq i \leq r$, then both the conjugacy classes $C_1$ and $C_2$ are distinct $z$-classes in $A_n$. Else $C_1 \cup C_2$ form a single $z$-class in $A_n$.

Now, when $C$ does not split, it follows from Proposition~\ref{main-prop-non-split}, that except the partition $1^3\nu$ where $\nu$ is a partition of $n-3$, with all parts odd and distinct without $1$ as its part, all conjugacy classes remain distinct $z$-classes. And in the case when $\lambda=1^3\nu$ its $z$-class can coincide with that of $3^1\nu$.

\section{Rational-valued Characters of $A_n$}\label{rational-charactersAn}
We begin with recalling characters of the alternating group from~\cite{pr}. First we note that, the number of partitions of $n$ with distinct and odd parts is equal to the number of self-conjugate partitions of $n$ (see Lemma 4.6.16 in~\cite{pr}). In fact, these are in one-one correspondence via folding. This corresponds to the split conjugacy classes. The complex irreducible characters of $A_n$ are given as follows (see Theorem 4.6.7 and 5.12.5 in~\cite{pr}). For every partition $\mu$ of $n$ which is not self-conjugate (this corresponds to non-split conjugacy classes), the irreducible character $\chi_{\mu}$ of $S_n$ restricts to an irreducible character of $A_n$. Since all characters of $S_n$ are integer-valued, these characters of $A_n$ are rational-valued too. Now, for all partitions $\mu$ of $n$ which are self-conjugate (these correspond to split conjugacy classes), there exists a pair of irreducible characters $\chi_{\mu}^+$ and $\chi_{\mu}^{-}$. The character values are given by the following formula.  When $g\in A_n$ of cycle type $\lambda$ with all parts distinct and odd, say, $\lambda=(2m_1+1, \cdots, 2m_l+1)$, and the folding corresponding to $\lambda$ is the partition $\mu$ then 
$$\chi_{\mu}^{\pm}(g_{\lambda}^+)=\frac{1}{2}\left(e_{\lambda}\pm \sqrt{e_{\lambda} |Z_{\lambda}|}\right)$$
and $ \chi_{\mu}^{\pm}(g_{\lambda}^-)= \chi_{\mu}^{\mp}(g_{\lambda}^+)$. Here $g_{\lambda}^+$ and $g_{\lambda}^-$ denote the two split conjugacy classes in $A_n$ and $e_{\lambda}=(-1)^{\sum_{i=1}^l m_i}$ . 
Else $\chi_{\mu}^+(g)=\chi_{\mu}^-(g) = \frac{\chi_{\mu}(g)}{2}$. Clearly the characters $\chi_{\mu}^{\mp}$ are rational valued if and only if $e_{\lambda}=1$ and $|Z_{\lambda}|$ is a perfect square.
\begin{lemma}
For $\lambda=(2m_1+1,\cdots,2m_l+1)$, if $|Z_{\lambda}|$ is a square then $e_{\lambda}=1$. 
\end{lemma}
\begin{proof}
In this case, $|Z_{\lambda}|=\prod_{i=1}^l (2m_i+1)=(2a+1)^2$ for some $a$. Then $\sum_{i=1}^l m_i$ must be even.
\end{proof}
\begin{proof}[\bf{Proof of Theorem~\ref{rational-character}}]
From the discussion above, all characters of $A_n$ corresponding to non-split conjugacy classes are rational-valued. And, both characters corresponding to split conjugacy classes are simultaneously rational-valued if and only if the partition $\lambda$ has all its parts distinct and odd and the product of parts is a perfect square. Clearly this is same as the criteria determining conjugacy classes which are rational.    
\end{proof}

\section{Some GAP calculations}\label{sum-squares}
In this work, we have come across two functions on natural numbers. 
The first one is $\epsilon$ defined as $$\epsilon(n)=\left|\{n=m_1^2+m_2^2+\cdots + m_r^2 \mid 1\leq m_1< m_2 <\ldots < m_r\leq n, m_i {\ \rm odd\ } \forall i\}\right|$$
and its generating function is $\displaystyle\prod_{i=0}^{\infty} \left(1+ x^{(2i+1)^2}\right)$.
And another one is $\delta$ defined as, 
$$\delta(n)=\left|\left\{n=n_1 + \cdots + n_r \mid 1\leq n_1< \ldots < n_r\leq n, n_i {\ \rm odd\ } \forall i, \displaystyle \prod_{i=1}^r n_i \in \mathbb N^2 \right\}\right|.$$
Writing a natural number as a sum of squares is well studied problem in number theory. However, we could not find references to these functions. Clearly $\epsilon(n) \leq \delta(n)$. The inequality could be strict, for example, $n=78=3 + 75$ where $3.75=15^2$ but none of the components are square. This happens infinitely often. For example, let  $p_1$ and $p_2$ be odd and distinct primes. Consider, $n=p_1+p_2+p_1p_2$ and the partition of $n$ given by $p_1^1p_2^1(p_1p_2)^1$. Then $\epsilon(n)<\delta(n)$. We may also consider, for example, $m=p_1+p_1p_2^2$, i.e., we have the partition of $m$ given as $p_1^1(p_1p_2^2)^1$. Then $\epsilon(m) < \delta(m)$.
We make a table for the values of $\epsilon$ and $\delta$ for small values of $n$ and also note down the partitions giving rise to the function $\delta$. Some values of $\delta(n)$ are also given in~\cite{br}.
\vskip3mm
\begin{center}
\begin{tabular}{|c|c|c|c||c|c|c|c|}
\hline
$n$ & $\epsilon(n)$ & $\delta(n)$  & partitions & $n$ & $\epsilon(n)$ & $\delta(n)$  & partitions\\  
 \hline
$9$ & $1$ & $1$ & $9^1$ &   $34$ & $1$ & $1$ &$9^125^1$  \\ \hline
$10$& $1$ & $1$ & $1^19^1$ &$35$ & $1$  &$1$ & $1^19^125^1$ \\ \hline
$23$& $0$ &$1$ &$3^1 5^1 15^1$& $39$& $0$ &$1$ &$3^1 9^1 27^1$ \\ \hline
$24$&$0$ &$1$&$1^1 3^1 5^1 15^1$ & $40$&$0$&$2$ &$1^13^19^127^1, 3^17^19^121^1$ \\ \hline
$25$ & $1$ & $1$ & $25^1$ & $41$ & $0$ &$1$ & $1^13^17^19^121^1$ \\ \hline
$26$ &$1$ &$1$&$1^125^1$ &$47$ &$0$&$3$ & $3^111^133^1, 5^17^135^1, 5^115^127^1$ \\ \hline
$30$&$0$&$1$ &$3^127^1$&$48$ &$0$&$5$&$1^13^111^133^1, 1^15^17^135^1$, \\ &&&&&&&$1^15^115^127^1$,$5^17^115^121^1$, \\ &&&&&&& $3^15^115^125^1$ \\ \hline
$31$&$0$&$2$&$1^13^127^1, 3^17^121^1$ &$49$&$1$&$3$&$49^1, 1^13^15^115^125^1$, \\ &&&&&&&$1^15^17^115^121^1$ \\ \hline
$32$ &$0$ &$2$ &$1^13^17^121^1, 3^15^19^115^1$ &$50$ &$1$ &$2$ &$1^149^1, 5^145^1$ \\ \hline
$33$&$0$&$1$&$1^13^15^19^115^1$&$51$&$0$&$1$& $1^15^145^1$ \\ \hline
\end{tabular}
\end{center}
\vskip3mm
Next, we used GAP~\cite{GAP4} to compute $z$-classes, rational conjugacy classes etc. Here we have some examples for $A_n$ which verifies our theorem.
\vskip3mm
\begin{center}
\begin{tabular}{|c|c|c|c|}
\hline
$n$ & number of  & number of  & partitions \\ 
    &conj. classes & $z$-classes & \\ \hline 
20 & 324 & 315 & $\{1^33^15^19^1, 3^25^19^1\}, \{1^13^17^19^1, 1^13^17^19^1\}$ \\
   &      &    & $\{1^13^15^111^1, 1^13^15^111^1\}, \{9^111^1,9^111^1 \}, \{1^119^1, 1^119^1\}$ \\
   &      &    & $\{7^113^1,7^113^1 \}, \{5^115^1, 5^115^1\}, \{1^317^1, 3^117^1, 3^117^1\}$ \\\hline
27 & 1526 & 1506 & $\{1^33^15^17^19^1, 3^25^17^19^1\}, \{1^13^15^17^111^1, 1^13^15^17^111^1 \}$ \\
   &     &   & $\{7^19^111^1, 7^19^111^1\}, \{5^19^113^1, 5^19^113^1\}$ \\
   &     &   & $\{1^311^113^1, 3^111^113^1, 3^111^113^1\}, \{5^17^115^1, 5^17^115^1 \}$\\
   &     &    &$\{1^39^115^1, 3^19^115^1, 3^19^115^1 \}, \{ 1^1 11^1 15^1, 1^1 11^1 15^1\}$ \\
   &     &    &$\{1^37^117^1, 3^17^117^1, 3^17^117^1 \}, \{ 1^35^119^1, 3^15^119^1, 3^15^119^1\} $\\
   &     &    &$\{1^17^119^1,  1^17^119^1\}, \{1^19^117^1,1^19^117^1 \}$\\
   &&& $\{1^33^121^1, 3^2 21^1\}, \{1^15^121^1, 1^15^121^1\}$\\
   &  &  & $\{1^13^123^1, 1^13^123^1\}, \{27^1, 27^1\}$\\\hline
\end{tabular}
\end{center}
\vskip3mm
The last column combines together the partitions which give same $z$-class and the repetition of a partition indicates a split conjugacy class.


\end{document}